\documentclass[10pt,reqno,a4paper,oneside,11pt]{amsart}

\usepackage{amsfonts}
\usepackage{mathrsfs}
\usepackage{amssymb}
\usepackage{amsmath}
\usepackage{amsthm}
\usepackage{graphicx}
\usepackage{color}
\usepackage[numbers,sort&compress]{natbib}
\usepackage{hyperref} 

\providecommand{\U}[1]{\protect\rule{.1in}{.1in}}

\newtheorem{Thm}{Theorem}[section]
\newtheorem{Cor}[Thm]{Corollary}
\newtheorem{Lem}[Thm]{Lemma}

\newtheorem{Def}{Definition}[section]
\newtheorem{remark}[Thm]{Remark}

\theoremstyle{definition}
\theoremstyle{remark}
\numberwithin{equation}{section}

\ifx\pdfoutput\relax\let\pdfoutput=\undefined\fi
\newcount\msipdfoutput
\ifx\pdfoutput\undefined\else
\ifcase\pdfoutput\else
\ifx\paperwidth\undefined\else
\ifdim\paperheight=0pt\relax\else\pdfpageheight\paperheight\fi
\ifdim\paperwidth=0pt\relax\else\pdfpagewidth\paperwidth\fi
\fi\fi\fi

\begin{document}
\pagestyle{myheadings}

\begin{center}
{\Large\textbf{The generalized hand-eye calibration matrix equation $AX-YB=C$ over dual quaternions}}	
\footnote{This research is supported by the National Natural
Science Foundation of China [grant number 12371023, 12271338].

\par
{}* Corresponding author.
\par  Email address: wqw@t.shu.edu.cn; wqwshu9@126.com (Q.W. Wang), xlm1817@163.com (L.M. Xie), hzh19871126@126.com (Z.H. He). 
}

\bigskip

{ \textbf{Lv-Ming Xie$^{a}$, Qing-Wen Wang$^{a,b,*}$, Zhuo-Heng He$^{a,b}$}}

{\small
\vspace{0.2cm}

$a$. Department of Mathematics and Newtouch Center for Mathematics, Shanghai University, Shanghai 200444, People's Republic of China\\
$b$. Collaborative Innovation Center for the Marine Artificial Intelligence, Shanghai University, Shanghai 200444, People's Republic of China

}

\end{center}

\vspace{0.4cm}

\begin{quotation}
\noindent\textbf{Abstract:}
In the field of robotics research, a crucial applied problem is the hand-eye calibration issue, which involves solving the matrix equation $AX = YB$. However, this matrix equation is merely a specific case of the more general dual quaternion matrix equation $AX-YB=C$, which also holds significant applications in system and control theory. Therefore, we in this paper establish the solvability conditions of this generalized hand-eye calibration dual quaternion matrix equation and provide a general expression for its solutions when it is solvable.  As an example of applications, we design a scheme for color image encryption and decryption based on this dual quaternion matrix equation. From the experiment, it can be observed that the decrypted images are almost identical to the original images. Therefore, the encryption and decryption scheme designed using this dual quaternion matrix equation is highly feasible.
\vspace{3mm}

\end{quotation}

\section{\textbf{Introduction}}

\par\setlength{\parindent}{2em}  Following Hamilton's identification of quaternions in 1843 \cite{Hamilton}, quaternions have found widespread applications in various fields such as signal processing, color image processing, and quantum physics (e.g., \cite{Cyrus,Bihan,Assefa,Pei,Graydon}). Consequently, research on quaternions has experienced rapid development. Subsequently, concepts like split quaternions, dual quaternions, and commutative quaternions emerged. Among these, dual quaternions, proposed by Clifford in 1873 \cite{Clifford}, play a crucial role not only in engineering domains like unmanned aerial vehicles control and small satellite systems, but also have extensive applications in computer vision, robotics, 3D motion modeling, and other fields, as evidenced by references \cite{YuanUAV,WangBJ,Figueredo,WangBJ1,KenrightBJ,ChengBJ,DaniilidisBJ,Wang2024}. This has garnered significant attention from scholars.

\par\setlength{\parindent}{2em}In the field of robotics research, a significant application challenge is the hand-eye calibration problem, which has consistently garnered attention. In 1994, Zhuang et al. \cite{Zhuang} formulated the hand-eye calibration problem as a solution to the matrix equation 
\begin{equation}\label{eq1.2}
	AX=YB.
\end{equation} As research progressed, it was discovered that employing dual quaternions, which can simultaneously represent rotation and translation, allows the transformation of the corresponding matrix equation \eqref{eq1.2} for the hand-eye calibration problem into a dual quaternion equation represented by $q_{A}q_{X}=q_{Y}q_{B}$, as detailed in the references \cite{Chen1,Li1,Zhao,Qi202307}. 
Evidently, the dual quaternion equation $q_{A}q_{X}=q_{Y}q_{B}$ constitutes a specific case of the dual quaternion matrix equation \eqref{eq1.2}, which is encompassed by the generalized hand-eye calibration dual quaternion matrix equation 
\begin{equation}\label{eq1.1}
	AX-YB=C.
\end{equation}
We are aware that when dealing with low-rank approximation problems for large-scale data, matrix decomposition can be employed, and the matrix equation \eqref{eq1.1} plays a crucial role in the matrix decomposition of the dual matrix \cite{Xu}. Furthermore, the matrix equation \eqref{eq1.1}, which is also a generalized Sylvester-type matrix equation, has important applications in the regulator problem with internal stability (RPIS) \cite{Cheng1977} and the non-interacting control problem \cite{Woude}. Additionally, the special case of the matrix equation \eqref{eq1.1}, $AX+YB=I$, is closely related to the algebraic regulator problem \cite{Bengtsson} and the output regulation problem with internal stability (ORIS) \cite{Wolovich}. 
Due to its significant applications in system and control theory, the matrix equation \eqref{eq1.1} has been a subject of ongoing research. For example, in 1952, Roth \cite{Roth} provided the solvability conditions for the matrix equation \eqref{eq1.1}. In 1977, Flanders et al. \cite{Flanders} presented an alternative method to Roth's proof for establishing the solvability of matrix equation \eqref{eq1.1} under certain conditions. Subsequently, in 1979, Baksalary et al. \cite{Baksalary} introduced another necessary and sufficient condition for the solvability of matrix equation \eqref{eq1.1}. Following this, the matrix equation \eqref{eq1.1} and its generalized forms have been extensively investigated, as documented in \cite{Emre,Guralnick,ZAK,Wimmer,Dmytryshyn,AnK}. Progress in the research on analytical solutions to various types of quaternion matrix equations, including dual quaternions, can be found in references \cite{Xie2023,Chen2024,Zhang2024,Liu2019,Liu2020,Liu20202}.

\par\setlength{\parindent}{2em}In addition, in the era of big data, encrypting crucial information is imperative, especially when it comes to image data. Considering that a quaternion matrix can represent a color image, we contemplate employing a dual quaternion matrix to simultaneously represent two color images. From this perspective, utilizing the dual quaternion matrix equation \eqref{eq1.1} enables the construction of encryption and decryption processes for color images. Since the generalized Sylvester matrix equation enjoys wide-ranging applications and in order to enrich its theoretical framework, we are now investigating the solutions of the dual quaternion matrix equation \eqref{eq1.1}. Based on our current knowledge, this is the first time solving a generalized Sylvester-type matrix equation over the dual quaternion algebra. 

\par\setlength{\parindent}{2em}\par\setlength{\parindent}{2em}An outline of this paper is as follows. In Section 2, we provide definitions for quaternions, dual quaternions, and dual quaternion matrices, along with  delineating essential operations related to them. Additionally, we also present some key lemmas. In Section 3, our focus is on establishing the necessary and sufficient conditions for the solvability of the dual quaternion matrix equation \eqref{eq1.1}. We also derive a general expression for the solutions when the matrix equation \eqref{eq1.1} is solvable. Furthermore, we investigate the solutions to the dual quaternion matrix equation \eqref{eq1.2} and explore various special cases of the matrix equation \eqref{eq1.1}. In Section 4, we present the application of the generalized Sylvester-type dual quaternion matrix equation \eqref{eq1.1} in the encryption and decryption process of color images. Finally, we present a summary of the key content of this paper in Section 5.
\par\setlength{\parindent}{2em}We begin by providing a concise overview of the notations utilized throughout this paper. Assume that the real number field, the quaternion skew field, and dual quaternions are denoted by $\mathbb{R}$, $\mathbb{H}$, and $\mathbb{DQ}$, respectively. The collection of all $k \times n$ matrices over $\mathbb{DQ}$ is denoted as $\mathbb{DQ}^{k \times n}$, and similarly, $\mathbb{H}^{k \times n}$ denotes the set of all $k \times n$ matrices over $\mathbb{H}$. If $A\in\mathbb{H}^{k \times n}$, the notation $r(A)$ denotes the rank of the matrix $A$.

\section{\textbf{ Preliminary}}

\par\setlength{\parindent}{2em}In this section, we review the definitions of quaternions, dual quaternions, and dual quaternion matrices, outlining essential operations associated with them. Additionally, we introduce crucial lemmas that form the basis for deriving the main results.

\subsection{Quaternions}
\par\setlength{\parindent}{2em}A quaternion $b$ can be expressed in the form $b=b_{0}+b_{1}\mathbf{i}+b_2\mathbf{j}+b_3\mathbf{k}$, where $b_{0},b_{1},b_2,b_3\in \mathbb{R}$ and $\mathbf{i},\mathbf{j},\mathbf{k}$ represent the imaginary components of quaternions, satisfying 
$
\mathbf{i}^2=\mathbf{j}^2=\mathbf{k}^2=-1, \mathbf{ij}=-\mathbf{ji}=\mathbf{k},\mathbf{jk}=-\mathbf{kj}=\mathbf{i},
$ and $\mathbf{ki}=-\mathbf{ik}=\mathbf{j}$. The following are some basic operations regarding quaternions.
For $a=a_{0}+a_{1}\mathbf{i}+a_2\mathbf{j}+a_3\mathbf{k},b=b_{0}+b_{1}\mathbf{i}+b_2\mathbf{j}+b_3\mathbf{k}\in\mathbb{H}$, then we have
$$
\begin{aligned}
	a+b&=(a_{0}+b_{0})+(a_{1}+b_{1})\mathbf{i}+(a_2+b_{2})\mathbf{j}+(a_3+b_{3})\mathbf{k}=b+a,\\
	ab&=(a_0 b_0-a_1 b_1-a_2 b_2-a_3 b_3)+(a_0 b_1+a_1 b_0+a_2 b_3
	-a_3 b_2)\mathbf{i}
	+(a_0 b_2\\
	&+a_2 b_0+a_3 b_1-a_1 b_3)\mathbf{j}+(a_0 b_3+a_3 b_0+a_1 b_2-a_2 b_1)\mathbf{k}.
	\\	ba&=(a_0 b_0-a_1 b_1-a_2 b_2-a_3 b_3)+(a_0 b_1+a_1 b_0-a_2 b_3+a_3 b_2)\mathbf{i}+(a_0 b_2\\
	&+a_2 b_0-a_3 b_1+a_1 b_3)\mathbf{j}+(a_0 b_3+a_3 b_0-a_1 b_2+a_2 b_1)\mathbf{k}.	
\end{aligned}
$$
It is evident that $ab\neq ba$.
\par\setlength{\parindent}{2em}Having established fundamental definitions for quaternions, we can now extend our understanding by introducing basic definitions for dual quaternions in a similar fashion.
\subsection{Dual quaternions}
The set of dual quaternions is defined as
$$
\mathbb{DQ}=\{d=d_{0}+d_{1}\epsilon:d_{0},d_{1}\in \mathbb{H} \},
$$
where $d_{0},d_{1}$ represent the standard part and the infinitesimal part of $d$, respectively.  The infinitesimal unit $\epsilon$ satisfies $\epsilon^2=0$ and commutes with multiplication in real numbers, complex numbers, and quaternions. Qi et al. \cite{Qi2022} has provided a partial order relation for dual numbers, along with some fundamental operations on dual quaternions. Suppose that $p=p_{0}+p_{1}\epsilon, q=q_0+q_1\epsilon\in\mathbb{DQ}$, then we have
$$
\begin{aligned}
	p+q &=(p_{0}+q_{0})+(p_{1}+q_{1})\epsilon,\\
	pq &=p_0q_0+(p_0q_1+p_1q_0)\epsilon.
\end{aligned}
$$
Since quaternions are non-commutative under multiplication, dual quaternions are also non-commutative under multiplication.
\begin{remark}
	Due to the property $\epsilon^2=0$, a non-zero dual quaternion is not necessarily invertible, implying that dual quaternions have zero divisors. However, every non-zero quaternion does have an inverse.
\end{remark}
\subsection{Dual quaternion matrices}
\par\setlength{\parindent}{2em} We denote a dual quaternion matrix as $C=C_0+C_1\epsilon\in\mathbb{DQ}^{k\times l}$,  with $C_0$ and $C_1$ belonging to the quaternion matrix set $\mathbb{H}^{k\times l}$. Due to the differences between dual quaternion matrices and matrices over conventional number fields, we first present the definition of equality for two dual quaternion matrices.
\begin{Def}\emph{\cite{Xie2023}}\label{de2.2.1}
	If $A=A_0+A_1\epsilon\in\mathbb{DQ}^{m\times n},B=B_0+B_1\epsilon\in\mathbb{DQ}^{m\times n}$, then
	$$
	A=B\Longleftrightarrow A_0=B_0,\ A_1=B_1.
	$$
\end{Def}
In a manner akin to dual quaternions, we can provide addition and multiplication operations for dual quaternion matrices. Assume that $B=B_0+B_1\epsilon\in\mathbb{DQ}^{l\times n}, C=C_0+C_1\epsilon\in\mathbb{DQ}^{k\times l}, D=D_0+D_1\epsilon\in\mathbb{DQ}^{k\times l},$ then we have
$$
\begin{aligned}
	C+D=(C_{0}+D_{0})+(C_{1}+D_{1})\epsilon,\\
	CB=C_0B_0+(C_0B_1+C_1B_0)\epsilon.
\end{aligned}
$$
For $C \in \mathbb{DQ}^{m \times n}$, there are additional operations such as $C^{*} , C^{\eta^*}$, as detailed in reference \cite{Xie2023}.
\subsection{Some significant lemmas}
\par\setlength{\parindent}{2em}
In the preceding subsection, definitions and relevant operations of dual quaternions and dual quaternion matrices were presented. In order to investigate the analytical solutions of the dual quaternion matrix equation \eqref{eq1.1}, we first introduce some crucial lemmas. For the sake of convenience in the description, we need some notations. The Moore-Penrose inverse of a quaternion matrix $A$ is defined as $A^\dagger$, and it follows these equations:
$$
AXA=A,  X A X=X, (A X)^{*}=A X, (X A)^{*}=X A.
$$
Furthermore, the projectors $I - A^\dagger A$ and $I - A A^\dagger$ are represented by the symbols $L_A$ and $R_A$, respectively. It is clear that
$$
R_A=\left(R_A\right)^2=R_A^{\dagger}, \
L_A=\left(L_A\right)^2=L_A^{\dagger}.
$$
\begin{Lem}\emph{\cite{Xie2023}}\label{lem2.1}
	Assume that $A=A_0+A_1\epsilon$ and $B=B_0+B_1\epsilon$ are given with appropriate sizes over $\mathbb{DQ}$, set
	$$
		A_{2}=A_1L_{A_0}, B_{2}=B_1-A_1A_0^{\dagger}B_0, C_{11}=R_{A_0}B_{2}, A_3=R_{A_0}A_{2}.
	$$
	Then the dual quaternion matrix equation $AX=B$ is solvable if and only if
	$$
	R_{A_0}B_0=0,\ R_{A_3}C_{11}=0
	$$
	or
	$$r\left[\begin{matrix}
		A_0&B_0
	\end{matrix}\right]=r(A_0), \ \ r\left[\begin{matrix}
		A_0 & B_{1} &A_1\\
		0 & B_0 & A_0
	\end{matrix}\right]=r\left[\begin{matrix}
		A_0 & A_{1}\\
		0 & A_0
	\end{matrix}\right].$$
	In this case, the general solution can be expressed as $X=X_0+X_1\epsilon$, where
	$$
	\begin{array}{l}
		X_0=A_0^{\dagger}B_0+L_{A_0}U, \\
		X_1=A_0^{\dagger}(B_{2}-A_{2}U)+L_{A_0}U_1,\\
		U=A_3^{\dagger}C_{11}+L_{A_3}U_2,
	\end{array}
	$$
	and $U_i (i = 1,2)$ are arbitrary matrices over $\mathbb{H}$.
\end{Lem}
\begin{Lem}\emph{\cite{Xie2020}}\label{lem2.2}
	Suppose that $A_0,A_1,B_0,B_1$  and $C$ are given over $\mathbb{H}$, denote
	$$
	A=R_{A_0}A_1, A_2=R_{A_0}C, B_2=B_0L_{B_1}, B_3=CL_{B_1}.
	$$
	Then the matrix equation 
		$	A_0X_0B_0+A_0X_1B_1+A_1X_2B_1=C$ 
	is solvable if and only if 
	$$
	R_{A}A_2 =0, R_{A_0}B_3 =0,B_3L_{B_2}=0,
	$$
	or
	$$
	\begin{array}{c}
		r\begin{pmatrix}
			A_0 &A_1&C
		\end{pmatrix}=r\begin{pmatrix}
			A_0 &A_1
		\end{pmatrix},r\begin{pmatrix}
			B_1 &0\\
			C &A_0
		\end{pmatrix}=r(B_1)+r(A_0), \\
		r\begin{pmatrix}
			C\\
			B_0\\
			B_1
		\end{pmatrix}=r\begin{pmatrix}
			B_0\\
			B_1
		\end{pmatrix}.	
	\end{array}
	$$
	In this case, the general solution can be expressed as
	$$
	\begin{array}{l}
		X_0=A_0^{\dagger}B_3B_2^{\dagger}+L_{A_0}U_5+U_6R_{B_2},\\
		X_1=A_0^{\dagger}(C-A_0X_0B_0-A_1X_2B_1)B_1^{\dagger}+L_{A_0}U_1+U_2R_{B_1},\\
		X_2=A^{\dagger}A_2B_1^{\dagger}+L_{A}U_3+U_4R_{B_1},
	\end{array}
	$$
	where $U_i(i=1,\cdots,6)$ represent arbitrary matrices over $\mathbb{H}$ with the suitable dimensions.
\end{Lem}
\begin{Lem}\emph{\cite{Xie2023}}\label{lem2.3}
	If $A, M, N, F$ and $K$ are given with appropriate sizes over $\mathbb{H}$, then
	$$
	r\left[\begin{array}{cc}
		A & M L_F \\ R_K N & 0
	\end{array}\right]=
	r\left[\begin{array}{ccc}
		A & M & 0 \\ N & 0 & K \\ 0 & F & 0
	\end{array}\right]-r(K)-r(F).
	$$
\end{Lem}

\section{The general solution of \eqref{eq1.1}}
\par\setlength{\parindent}{2em}
By employing the Moore-Penrose inverses and ranks of matrices, we can use the previously established lemmas to deduce the conditions for the solvability of the dual quaternion matrix equations \eqref{eq1.1}. We also offer an expression for the general solution of \eqref{eq1.1}.  Now, we present the serious theorem of this paper.
\begin{Thm}\label{th3.1}
	Suppose that $A=A_0+A_1\epsilon\in\mathbb{DQ}^{n\times k},B=B_0+B_1\epsilon\in\mathbb{DQ}^{l\times m},$ and $C=C_0+C_1\epsilon\in\mathbb{DQ}^{n\times m}$ are given, set
	$$
	\begin{array}{l}
		A_{11}=A_1L_{A_0},A_2=R_{A_0}A_{11}, A_3=R_{A_0},C_3=-R_{A_0}C_0,
		B_2=R_{B_0}B_1,\\
		A_4=R_{A_2}R_{A_0},A_5=-A_4A_1A_0^{\dagger},
		C_4=A_4(A_1A_0^{\dagger}C_0+C_0B_0^{\dagger}B_1-C_1),\\
		A_6=R_{A_4}A_5, C_5=R_{A_4}C_4,B_3=B_2L_{B_0},B_4=C_4L_{B_0}.
	\end{array}
	$$
	Then the following descriptions are equivalent:
	\begin{enumerate}
		\item The dual quaternion matrix equation \eqref{eq1.1} is consistent.
		\item 
			$	C_3L_{B_0}=0,\ B_4L_{B_3}=0.$
		\item The ranks obey:
			\begin{equation}
				r\left[\begin{matrix}
					B_0 & 0 \\
					-C_0& A_0 
				\end{matrix}\right]=r(B_0)+r(A_0),  
				\label{eq3.2}
			\end{equation}
			\begin{equation}
				r\left[\begin{matrix}
					B_0 & 0 & 0 & 0\\
					B_1 & B_0 & 0 &0\\
					-C_1 & -C_0 & A_0 &A_1\\
					0 & 0 & 0 &A_0
				\end{matrix}\right]=r\begin{bmatrix}
					B_0 & 0\\ B_1 & B_0
				\end{bmatrix}+r\begin{bmatrix}
					A_0 & A_1\\ 0 & A_0
				\end{bmatrix}.
				\label{eq3.3}
			\end{equation}
	\end{enumerate}
	Under such conditions, the general solution of the matrix equation \eqref{eq1.1} can be expressed as $X=X_0+X_1\epsilon$ and $Y=Y_0+Y_1\epsilon$, where
	$$\begin{aligned}
			X_0&=A_0^{\dagger}(C_0+Y_0B_0)+L_{A_0}W,
			\\
			X_1&=A_0^{\dagger}[C_1+Y_0B_1+Y_1B_0-A_1A_0^{\dagger}(C_0+Y_0B_0)-A_{11}W]+L_{A_0}W_1,\\
			W&=A_2^{\dagger}A_3[C_1+Y_0B_1+Y_1B_0-A_1A_0^{\dagger}(C_0+Y_0B_0)]+L_{A_2}W_2,\\
			Y_0&=A_3^{\dagger}C_3B_0^{\dagger}+L_{A_3}U_1+U_2R_{B_0},
			\\
			Y_1&=A_4^{\dagger}(C_4-A_5U_1B_0-A_4U_2B_2)B_0^{\dagger}+L_{A_4}W_3+W_4R_{B_0},    \\
			U_1&=A_6^{\dagger}C_5B_0^{\dagger}+L_{A_6}W_5+W_6R_{B_0},\\
			\quad
			U_2&=A_4^{\dagger}B_4B_3^{\dagger}+L_{A_4}W_7+W_8R_{B_3},
	\end{aligned}
$$
	and the arbitrary matrices $W_i(i=1,\cdots,8)$ have the proper sizes.
\end{Thm}
\begin{proof}
	The proof is split into two parts.
	\par\setlength{\parindent}{2em}\textbf{Part 1.}
	The matrix equation \eqref{eq1.1} over $\mathbb{DQ}$  is equivalent to a system of matrix equations 
	\begin{equation}\label{eq2.3}
		\left\{\begin{array}{l}
			A_0X_0-Y_0B_0=C_0,\\
			A_0X_1+A_1X_0-Y_0B_1-Y_1B_0=C_1
		\end{array}\right.
	\end{equation}
	over $\mathbb{H}$, as per Definition \ref{de2.2.1} and the definition of dual quaternion matrix multiplication. As a result, it is possible to solve the dual quaternion matrix equation \eqref{eq1.1} by solving the system of quaternion matrix equations \eqref{eq2.3}, which is also being studied for the first time.
	\par\setlength{\parindent}{2em}\textbf{Part 2.}$(1)\Longleftrightarrow(2)$. It is evident that the system \eqref{eq2.3} can be expressed as
	\begin{equation}\label{eq3.11}
		\left\{\begin{array}{l}
			A_0X_0=C_0+Y_0B_0,\\
			A_0X_1+A_1X_0=C_1+Y_0B_1+Y_1B_0.
		\end{array}\right.
	\end{equation}
	Lemma \ref{lem2.1} leads us to the conclusion that  
	\begin{equation}\label{eq3.12}
		R_{A_0}(C_0+Y_0B_0)=0
	\end{equation}
	and
	\begin{equation}\label{eq3.13}
		R_{A_2}R_{A_0}[C_1+Y_0B_1+Y_1B_0-A_1A_0^{\dagger}(C_0+Y_0B_0)]=0	
	\end{equation}
	are the necessary and sufficient condition that makes the system \eqref{eq3.11} consistent. In this case, the general solution of the system \eqref{eq3.11} can be written as $X=X_0+X_1\epsilon$, where
	$$
	\begin{aligned}
		X_0&=A_0^{\dagger}(C_0+Y_0B_0)+L_{A_0}W,\\
		X_1&=A_0^{\dagger}[C_1+Y_0B_1+Y_1B_0-A_1A_0^{\dagger}(C_0+Y_0B_0)-A_{11}W]+L_{A_0}W_1,\\
		W&=A_2^{\dagger}A_3[C_1+Y_0B_1+Y_1B_0-A_1A_0^{\dagger}(C_0+Y_0B_0)]+L_{A_2}W_2,
	\end{aligned}
	$$
	and $W_1,W_2$ are arbitrary matrices over $\mathbb{H}$. Clearly, the quaternion matrix equation \eqref{eq3.12} is equivalent to
	\begin{equation}\label{eq3.14}
		A_3Y_0B_0=C_3.
	\end{equation}
	According to Lemma \ref{lem2.2}, we have the quaternion matrix equation \eqref{eq3.14} is solvable only if 
	$$
	R_{A_3}C_3=0, C_3L_{B_0}=0.
	$$
	Since $A_0^{\dagger}R_{A_0}=0$, it always holds true that $R_{A_3}C_3=0$. In this situation, the quaternion matrix equation \eqref{eq3.14} has a solution as long as $C_3L_{B_0}=0$, and the general solution can be represented by  
	\begin{equation}\label{eq3.15}
		Y_0=A_3^{\dagger}C_3B_0^{\dagger}+L_{A_3}U_1+U_2R_{B_0}.
	\end{equation}
	Substituting the equation \eqref{eq3.15} into the matrix equation \eqref{eq3.13}, we can determine that
	\begin{equation}\label{eq3.16}
		A_4U_2B_2+A_4Y_1B_0+A_5U_1B_0=C_4.
	\end{equation}
	By applying Lemma \ref{lem2.2}, we can derive the quaternion matrix equation \eqref{eq3.16} is solvable if and only if
	$$
	\begin{aligned}
		R_{A_6}C_5=0,\ R_{A_4}B_4=0,\ B_4L_{B_3}=0.
	\end{aligned}
	$$
	In this case, the general solution of the matrix equation \eqref{eq3.16} can be expressed as
	$$
	\begin{array}{l}
		Y_1=A_4^{\dagger}(C_4-A_5U_1B_0-A_4U_2B_2)B_0^{\dagger}+L_{A_4}W_3+W_4R_{B_0},\\
		U_1=A_6^{\dagger}C_5B_0^{\dagger}+L_{A_6}W_5+W_6R_{B_0},\\
		U_2=A_4^{\dagger}B_4B_3^{\dagger}+L_{A_4}W_7+W_8R_{B_3},
	\end{array}
	$$
	where $W_i(i=3,\cdots,8)$ denote arbitrary matrices of appropriate dimensions over $\mathbb{H}$. 
	\par\setlength{\parindent}{2em}
	In the subsequent proof, the verification of $R_{A_6}C_5=0,\ R_{A_4}B_4=0$ as identities will be provided, thereby ensuring that the matrix equation \eqref{eq3.16} is solvable if and only if $B_4L_{B_3}=0$. At this stage, the general solution expressions for the dual quaternion matrix equation \eqref{eq1.1} are given as $X=X_0+X_1\epsilon$ and $Y=Y_0+Y_1\epsilon$.
	\par\setlength{\parindent}{2em}$(2)\Longleftrightarrow(3)$. We need to show the equivalence between $C_3L_{B_0}=0$ and \eqref{eq3.2}, and between $B_4L_{B_3}=0$ and \eqref{eq3.3}, respectively. 
	\par\setlength{\parindent}{2em}Referring to Lemma \ref{lem2.3}, we obtain
	$$
		C_3L_{B_0}=0 \Leftrightarrow r(C_3L_{B_0})=0 \Leftrightarrow r\begin{bmatrix}
			B_0 & 0\\-C_0 & A_0
		\end{bmatrix}=r(A_0)+r(B_0).
	$$
	This indicates that $C_3L_{B_0}=0$ is equivalent to \eqref{eq3.2}. Now, we focus on establishing the equivalence between $B_4L_{B_3}=0$ and \eqref{eq3.3}, while also proving $R_{A_6}C_5=0, R_{A_4}B_4=0$ are two identities. The following statements are valid in light of block Gaussian elimination and Lemma \ref{lem2.3}.
	$$
	\begin{aligned}
		R_{A_6}C_5=0 &\Leftrightarrow r(R_{A_6}C_5)=0 \Leftrightarrow
		r\left[\begin{matrix}
			A_6 & C_5
		\end{matrix}\right]=r(A_6),\\
		&\Leftrightarrow r\left[\begin{matrix}
			A_4 & 0 &0
		\end{matrix}\right]=r\left[\begin{matrix}
			A_4 & 0
		\end{matrix}\right].
	\end{aligned}
	$$
	In the same way, we may show that
	$$	
	\begin{aligned}
		R_{A_4}B_4=0
		\Leftrightarrow r\left[\begin{matrix}
			A_4 & C_4 \\
			0 & B_0
		\end{matrix}\right]=r(A_4)+r(B_0)
		\Leftrightarrow r\left[\begin{matrix}
			A_4 &0 \\
			0   &B_0
		\end{matrix}\right]=r(A_4)+r(B_0).
	\end{aligned}$$
	It follows that the equations $R_{A_6}C_5=0$ and $R_{A_4}B_4=0$ are all identities. $B_4L_{B_3}=0$ is subjected to Lemma \ref{lem2.3}, yielding
		$$
		\begin{aligned}
			r(B_4L_{B_{3}})=0&\Leftrightarrow 
			r\left[\begin{matrix}
				B_4  \\
				B_{3} 
			\end{matrix}\right]=r(B_3) ,\\
			&\Leftrightarrow
			r\left[\begin{matrix}
				B_0 & 0 & 0 &0\\
				B_1 & B_0 & 0&0\\
				-C_{1} &-C_0& A_0 &A_1\\
				0&0&0&A_0
			\end{matrix}\right]=r\left[\begin{matrix}
				B_0 & 0 \\ B_1 & B_0
			\end{matrix}\right]+r\left[\begin{matrix}
				A_0&A_1\\0 & A_0
			\end{matrix}\right],
		\end{aligned}$$
	i.e., $B_4L_{B_3}=0$ $\Longleftrightarrow$ \eqref{eq3.3}.
\end{proof}
\par\setlength{\parindent}{2em}Utilizing Theorem \ref{th3.1} in various contexts, we present both necessary and sufficient conditions for the solvability of the dual quaternion matrix equations $YB=C$ and $AX=YB$, along with providing the general solution expressions for them.
\begin{Cor}\emph{\cite{Xie2023}}\label{cor3.1}
	Let $B=B_0+B_1\epsilon\in\mathbb{DQ}^{k\times l}$, and $C=C_0+C_1\epsilon\in\mathbb{DQ}^{n\times l}$. Denote
	$$
	\begin{array}{l}
		B_{2}=R_{B_0}B_1,\ C_2=C_1-C_0B_0^{\dagger}B_1, B_{00}=B_{2}L_{B_0},\ C_{00}=C_{2}L_{B_0}.
	\end{array}
	$$
	Then the dual quaternion matrix equation $YB=C$ is consistent if and only if
	$$
	C_0L_{B_0}=0,\ C_{00}L_{B_{00}}=0,
	$$
	or
	$$
	r\left[\begin{matrix}
		B_0 \\
		C_0
	\end{matrix}\right]=r(B_0),\ r\left[\begin{matrix}
		C_1 & C_0\\
		B_0 & 0\\
		B_1 & B_0
	\end{matrix}\right]=r\left[\begin{matrix}
		B_0 & 0\\
		B_1 & B_0
	\end{matrix}\right].
	$$
	In such circumstance, the general solution can be expressed as $Y=Y_0+Y_1\epsilon$, where
	$$
	\begin{array}{l}
		Y_0=C_0B_0^{\dagger}+WR_{B_0},\\
		Y_1=(C_{2}-WB_{2})B_0^{\dagger}+W_1R_{B_0},\\
		W=C_{00}B_{00}^{\dagger}+W_2R_{B_{00}},
	\end{array}
	$$
	and $W_i (i = 1,2)$ represent arbitrary matrices over $\mathbb{H}$ with the suitable dimensions.
\end{Cor}
\begin{Cor}\label{cor3.2}
	Assume that $A,B$ are given with the suitable dimensions over $\mathbb{DQ}$, set
	$$
	\begin{array}{l}
		A_{11}=A_1L_{A_0},A_2=R_{A_0}A_{11}, A_3=R_{A_0},
		B_2=R_{B_0}B_1,A_4=R_{A_2}R_{A_0},\\
		A_5=-A_4A_1A_0^{\dagger},
		A_6=R_{A_4}A_5,B_3=B_2L_{B_0}.
	\end{array}
	$$
	Then the dual quaternion matrix equation \eqref{eq1.2} is solvable, and the general solution can be expressed as $X=X_0+X_1\epsilon$ and $Y=Y_0+Y_1\epsilon$, where
	$$
	\begin{aligned}
			X_0&=A_0^{\dagger}Y_0B_0+L_{A_0}W,
			\\
			X_1&=A_0^{\dagger}[Y_0B_1+Y_1B_0-A_1A_0^{\dagger}Y_0B_0-A_{11}W]+L_{A_0}W_1, \\
			W&=A_2^{\dagger}A_3[Y_0B_1+Y_1B_0-A_1A_0^{\dagger}Y_0B_0]+L_{A_2}W_2,    \\
			Y_0&=L_{A_3}U_1+U_2R_{B_0},
			\\
			Y_1&=A_4^{\dagger}(-A_5U_1B_0-A_4U_2B_2)B_0^{\dagger}+L_{A_4}W_3+W_4R_{B_0}, \\
			U_1&=L_{A_6}W_5+W_6R_{B_0},\\
			\quad
			U_2&=L_{A_4}W_7+W_8R_{B_3},
	\end{aligned}
	$$
	and $W_i(i=1,\cdots,8)$ denote arbitrary matrices over $\mathbb{H}$.
\end{Cor}
\begin{remark}
	The previously mentioned hand-eye calibration model can be reformulated to address the dual quaternion equation $q_Aq_X=q_Yq_B$ \cite{Li1}, which is a special case of the dual quaternion matrix equation \eqref{eq1.2}.
\end{remark}
\section{An application}
In this section, we design an encryption and decryption scheme for color images based on generalized Sylvester-type dual quaternion matrix equation \eqref{eq1.1}. Furthermore, we conduct experiments with a set of color images to verify the feasibility of this scheme.
\par\setlength{\parindent}{2em}We know that a quaternion matrix can represent a color image, and since both the standard part and the infinitesimal part of a dual quaternion matrix are quaternion matrices, we can use a dual quaternion matrix to represent two color images. Now, we present a scheme for encrypting and decrypting color images using the generalized Sylvester-type dual quaternion matrix equation \eqref{eq1.1}, as shown in the figure below.
\begin{figure}[ht]
	\begin{center}
		\includegraphics[width=0.7\linewidth]{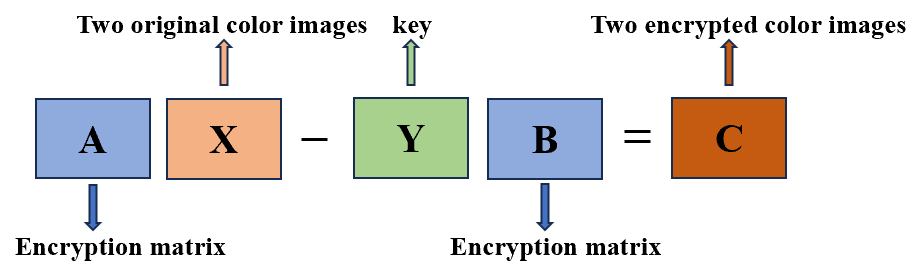}
		\caption{Encryption of two color images based on \eqref{eq1.1}}
		\label{fig1}
	\end{center}
\end{figure}
\par\setlength{\parindent}{0em}In Fig. \ref{fig1}, the dual quaternion matrices $A$ and $B$ form a password book. To ensure that the encrypted information is difficult to crack, it is required that the standard part and the infinitesimal part of the encryption matrices $A$ and $B$ are full-rank matrices either in rows or columns. The dual quaternion matrix $Y$, which also represents two color images, is the key used for both encryption and decryption.
\par\setlength{\parindent}{2em} We randomly select two color images as the original color images to be encrypted, and two other color images as the keys, as shown in Fig. \ref{fig2}.
\begin{figure}[!htp]
	\begin{center}
		\includegraphics[width=0.5\linewidth]{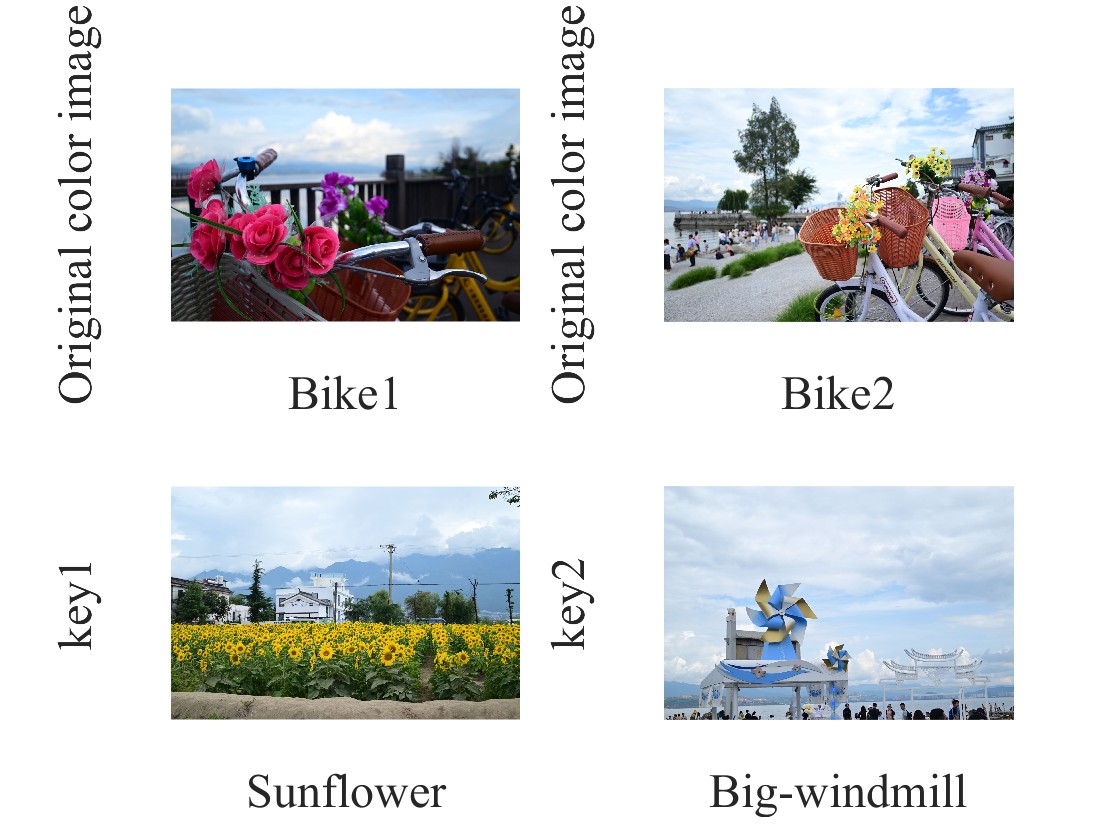}
		\caption{Encrypted two original color images and key}
		\label{fig2}
	\end{center}
\end{figure}
\par\setlength{\parindent}{2em}Based on the encryption principle shown in Fig. \ref{fig1}, we encrypt the color images ``bike1" and ``bike2". The encrypted images are shown below.
\begin{figure}[!htp]
	\begin{center}
		\includegraphics[width=0.5\linewidth]{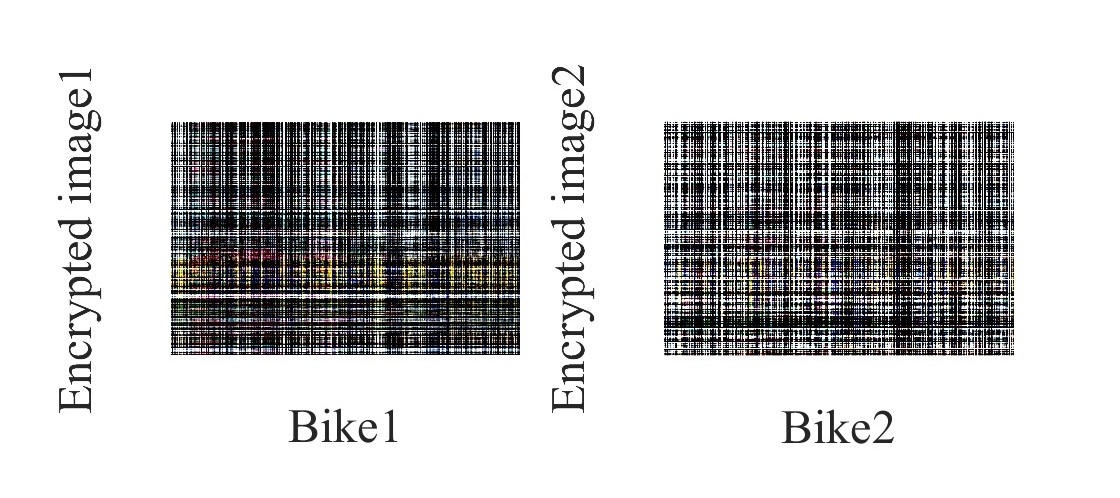}
		\caption{Two encrypted color images}
		\label{fig3}
	\end{center}
\end{figure}
\par\setlength{\parindent}{2em}Now, according to the encryption principle in Fig.\ref{fig1} and Theorem \ref{th3.1}, we decrypt the two color images in Fig.\ref{fig3}. The decrypted images are:
\begin{figure}[h]
	\begin{center}
		\includegraphics[width=0.5\linewidth]{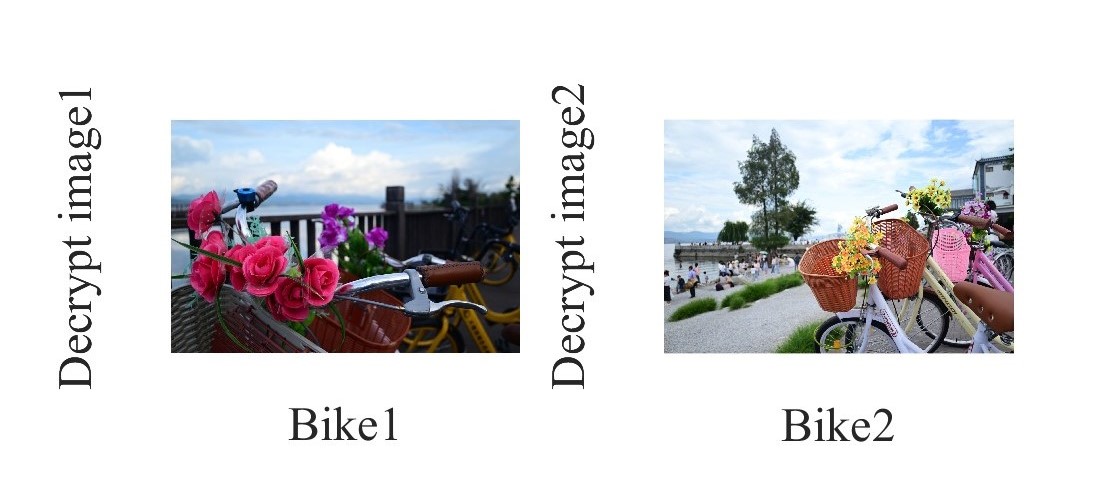}
		\caption{Two decrypted color images}
		\label{fig4}
	\end{center}
\end{figure}

\par\setlength{\parindent}{2em}Evidently, the decrypted images are visually identical to the original color images. To evaluate the quality of image restoration, we use the structural similarity index measure (SSIM). See the table below for details.
\begin{center}
	\begin{tabular}{c|c}
		\hline
		Color image name & SSIM \\ \hline
		Bike 1 & 0.99 \\ \hline
		Bike 2 & 0.99 \\ \hline
	\end{tabular}
\end{center}
From the table above, it can be observed that the decrypted images are almost identical to the original images. Therefore, the encryption and decryption scheme designed using generalized Sylvester-type dual quaternion matrix equation \eqref{eq1.1} is highly feasible.
\section{Conclusion}
In this paper, utilizing the Moore-Penrose inverses and ranks of matrices, we have provided necessary and sufficient conditions for the solvability of the dual quaternion matrix equation \eqref{eq1.1}. We have also presented a general expression for the solutions using the Moore-Penrose inverses of matrices when it is consistent. Furthermore, we have investigated the dual quaternion matrix equation \eqref{eq1.2} and presented a general expression for its solutions. Since the dual quaternion matrix equation \eqref{eq1.2} encompasses the dual quaternion equation $q_{A}q_{X}=q_{Y}q_{B}$, the analytical solutions of the matrix equation \eqref{eq1.2} may be beneficial for research on hand-eye calibration problems. Additionally, we have discussed some other special cases of the dual quaternion matrix equation \eqref{eq1.1}. Ultimately, we presented the application of the generalized Sylvester-type dual quaternion matrix equation \eqref{eq1.1} in color image processing and demonstrated through experiments that the scheme in Fig. \ref{fig1} is highly feasible.

\par\setlength{\parindent}{0em}\textbf{Conflicts of Interest}: The authors declare no conflicts of interest.


\begin{thebibliography}{99}


\bibitem {Hamilton} W.R. Hamilton,  {\it Lectures on quaternions}. Dublin: Hodges and Smith, 1853.

\bibitem {Cyrus} J. Cyrus, C.T. Clive and P.M. Danilo,  A class of quaternion valued affine projection algorithms. {\it Signal Process, }\textbf{93} (2013) 1712–1723.

\bibitem {Bihan} N. Le Bihan and J. Mars, Singular value decomposition of quaternion matrices: a new tool for vector-sensor signal processing. {\it Signal Process, }\textbf{84} (2004) 1177-1199.

\bibitem {Assefa}D. Assefa, L. Mansinha, K.F. Tiampo, H. Rasmussen and K. Abdella, Local quaternion Fourier transform and color image texture analysis. {\it Signal Process, } \textbf{90} (2010) 1825–1835.

\bibitem {Pei}S.C. Pei, J.H. Chang and J.J. Ding, Quaternion matrix singular value decomposition and its applications for color image processing. {\it Proceedings 2003 International Conference on Image Processing (Cat. No.03CH37429), }\textbf{1} (2003) I-805.

\bibitem {Graydon}M.A. Graydon,  {\it Quaternions and quantum theory (Master's thesis)}. University of Waterloo, Waterloo, Ontario, Canada,
2011.
\bibitem {YuanUAV} Y.X. Yuan and M. Ryll, `` Dual quaternion control of UAVs with cable-suspended load", 2024, arXiv:2404.07635.

\bibitem {WangBJ} X. Wang, C. Yu and Z. Lin,  ``A dual quaternion solution to attitude and position control for rigid body coordination", {\it IEEE transactions on robotics} \textbf{28} (2012) 1162-1170.

\bibitem{Figueredo} L.F. da Cruz Figueredo, B.V. Adorno and J.Y. Ishihara,   ``Robust $H_{\infty}$ kinematic control of manipulator robots using dual quaternion algebra", {\it Automatica} \textbf{132} (2021) 109817.

\bibitem {WangBJ1} X. Wang and H. Zhu, ``On the comparisons of unit dual quaternion and homogeneous transformation matrix", {\it Advances in	Applied Clifford Algebras} \textbf{24} (2014) 213-229.

\bibitem {KenrightBJ} B. Kenright, ``A beginners guide to dual-quaternions", {\it 20th International Conference in Central Europe on Computer Graphics, Visualization and Computer Vision,} Plzen, Czech, 2012.

\bibitem {Xie2020} M.Y. Xie and Q.W. Wang, ``The reducible solution to a quaternion tensor equation", {\it Frontiers of mathematics in China} \textbf{15} (2020) 1047-1070.

\bibitem {Zhuang} H.Q. Zhuang, Z.S. Roth and R. Sudhakar, ``Simultaneous robot/world and tool/flange calibration by solving homogeneous transformation equations of the form $AX=YB$", {\it IEEE transactions on robotics and automation} \textbf{10} (1994) 549-554.


\bibitem {Clifford} W.K. Clifford, ``Preliminary sketch of bi-quaternions", {\it Proceedings of the London Mathematical Society} \textbf{4} (1873) 381-395.

\bibitem {ChengBJ} J. Cheng, J. Kim, Z. Jiang and W. Che, ``Dual quaternion-based graph SLAM", {\it Robotics and Autonomous Systems} \textbf{77} (2016) 15-24.

\bibitem {DaniilidisBJ} K. Daniilidis, ``Hand-eye calibration using dual quaternions", {\it The International journal of robotics research}  \textbf{18} (1999) 286-298.

\bibitem {Wang2024} X. Wang, H.X. Sun, C.L. Liu and H.W. Song, ``Dual quaternion operations for rigid body motion and their application to the hand-eye calibration", {\it Mechanism and Machine Theory} \textbf{193} (2024) 105566.

\bibitem {Xu} R.J. Xu, T. Wei, Y.M. Wei and H. Yan, ``UTV decomposition of dual matrices and its applications", {\it Computational and applied mathematics} \textbf{43} (2024) 41.

\bibitem {Bengtsson} G. Bengtsson, ``Output regulation and internal models\textemdash a frequency domain approach", {\it Automatica} \textbf{13} (1977) 333-345. 

\bibitem {Cheng1977} L.F. Cheng and J.B. Pearson, ``Frequency domain synthesis of multivariable linear regulators", {\it IEEE transactions on automatic control} \textbf{23} (1978) 3-15. 

\bibitem {Wolovich} W.A. Wolovich and P. Ferreira, ``Output regulation and tracking in linear multivariable systems", {\it IEEE transactions on automatic control} \textbf{24} (1979) 460-465. 

\bibitem{Woude} J.W. Van der Woude, ``Almost non-interacting control by measurement feedback", {\it Systems and control letters} \textbf{9} (1987) 7-16.

\bibitem {Roth}  W.E. Roth, ``The equation $AX-YB=C$ and $AX-XB=C$ in matrices", {\it Proceedings of the American Mathematical Society} \textbf{3} (1952) 392–396.

\bibitem {Flanders} H. Flanders and H.K. Wimmer, ``On the matrix equations $AX - XB = C$ and $AX - YB = C$", {\it SIAM Journal on Applied Mathematics} \textbf{32} (1977) 707-710.

\bibitem {Baksalary} J.K. Baksalary and R. Kala, ``The matrix equations $AX-YB = C$", {\it Linear Algebra and Its Applications} \textbf{25} (1979) 41–43.

\bibitem {Emre} E. Emre and L.M. Silverman, ``The equation $XR + QY = \Phi $: a characterization of solutions", {\it SIAM journal on control and optimization} \textbf{19} (1981) 33-38.

\bibitem {Guralnick} R.M. Guralnick, ``Roth’s theorems for sets of matrices", {\it Linear Algebra and Its Applications} \textbf{71} (1985) 113-117.

\bibitem {ZAK} S.H. ZAK, ``On the polynomial matrix equation $AX+YB=C$", {\it IEEE transactions on automatic control} \textbf{30} (1985) 1240-1242.

\bibitem {Wimmer} H.K. Wimmer, ``Consistency of a pair of generalized Sylvester equations", {\it IEEE transactions on automatic control} \textbf{39} (1994) 1014-1016.

\bibitem {Chen1} Z.M. Chen, C. Ling, L.Q. Qi and H. Yan, ``A regularization-patching dual quaternion optimization method for solving the hand-eye calibration problem", \textit{ Journal of Optimization Theory and Applications} \textbf{200} (2024) 1193-1215.

\bibitem {Dmytryshyn} A. Dmytryshyn and B. Kagstr\"{o}m, ``Coupled Sylvester-type matrix equations and block diagonalization", {\it SIAM journal on matrix analysis and applications} \textbf{36} (2015) 580-593.

\bibitem {AnK} I.J. An, E. Ko and J.E. Lee, ``On the generalized Sylvester operator equation $AX-YB=C$", {\it Linear and Multilinear Algebra} (2022)  1-12.

\bibitem {Li1} A. Li, L. Wang and D. Wu, ``Simultaneous robot-world and hand-eye calibration using dual-quaternions and Kronecker product", {\it International Journal of the Physical Sciences} \textbf{5} (2010) 1530-1536.

\bibitem {Zhao} Z.J. Zhao, ``Simultaneous robot-world and hand-eye calibration by the alternative linear programming", {\it Pattern Recognition Letters} \textbf{127} (2019) 174-180.


\bibitem {Qi2022} L.Q. Qi, C. Ling and H. Yan, ``Dual quaternions and dual quaternion vectors", {\it Communications on Applied Mathematics and Computation} \textbf{4} (2022) 1494-1508.


\bibitem {Xie2023} L.M. Xie and Q.W. Wang, ``A system of dual quaternion matrix equations with its applications", 2023, arXiv:2312.10037.

\bibitem {Chen2024} Y. Chen, Q.W. Wang and L.M. Xie, ``Dual quaternion matrix equation $AXB = C$ with applications", {\it Symmetry} \textbf{16} (2024) 287.

\bibitem {Zhang2024} Y. Zhang, Q.W. Wang and L.M. Xie, ``The Hermitian solution to a new system of commutative
quaternion matrix equations", {\it Symmetry} \textbf{16} (2024) 361.

\bibitem {Qi202307} L.Q. Qi, ``Standard dual quaternion optimization and its applications in hand-eye calibration and SLAM", {\it Communications on Applied Mathematics and Computation} \textbf{5} (2023) 1469-1483.

\bibitem{Liu2019} X. Liu and Y. Zhang, ``Consistency of Split Quaternion Matrix Equations $AX^{\star }-XB=CY+D$ and $X-AX^\star B=CY+D$",
{\it{Advances in applied Clifford algebras }} \textbf{29} (2019) 64.

\bibitem{Liu2020} X. Liu and Y. Zhang, ``Least squares $X=\pm X^{\eta^\ast}$ solutions to split quaternion matrix equation $AXA^{\eta^\ast}=\ B$", {\it Mathematical Methods in the Applied Sciences} \textbf{43} (2020) 2189-2201.

\bibitem{Liu20202} X. Liu and Z.H. He, ``On the split quaternion matrix equation $AX = B$", {\it Banach Journal of Mathematical Analysis} \textbf{14} (2020) 228-248.


\end{thebibliography}
\end{document}